\documentclass[12pt, reqno]{amsart}
\usepackage[margin=1in]{geometry}
\usepackage{amsthm, amsmath, amsfonts, amssymb, mathabx, tikz}
\usepackage{enumerate}
\usepackage{hyperref}

\usepackage{color}

\newtheorem{theorem}{Theorem}[section]
\newtheorem{proposition}[theorem]{Proposition}
\newtheorem{lemma}[theorem]{Lemma}

\theoremstyle{definition}
\newtheorem{definition}[theorem]{Definition}
\theoremstyle{remark}
\newtheorem{remark}{Remark}

\newcommand{\Z}[0]{\mathbb{Z}}
\newcommand{\R}[0]{\mathbb{R}}
\newcommand{\Q}[0]{\mathbb{Q}}

\newcommand{\T}[0]{\mathbb{T}}
\newcommand{\Om}[0]{\Omega}

\newcommand{\bs}[0]{\setminus}

\newcommand{\vep}[0]{\varepsilon}
\newcommand{\supp}[0]{\operatorname{supp}}
\newcommand{\lsm}[0]{\lesssim}
\newcommand{\lsa}[0]{\lessapprox}

\newcommand{\wh}[1]{\widehat{#1}}
\newcommand{\wc}[1]{\widecheck{#1}}
\newcommand{\mc}[1]{\mathcal{#1}}

\newcommand{\wt}[1]{\widetilde{#1}}

\newcommand{\nms}[1]{\| #1 \|}

\newcommand{\Qp}{\mathbb Q_p}
\newcommand{\cO}{\mathbb{Z}_p}
\newcommand{\1}{\mathbf 1}
\newcommand{\norm}[1]{\lVert#1\rVert}

\newcommand{\I}{\mathcal{I}}
\newcommand{\B}{\mathcal{B}}

\newcommand{\cN}{\mathcal{N}}

\title{An improved high-low method for the parabola}
\author{Zane Kun Li}
\author{Po-Lam Yung}

\address{Department of Mathematics, North Carolina State University, Raleigh, NC 27695
}
\email{zkli@ncsu.edu}

\address{Mathematical Sciences Institute, Australian National University, Canberra, ACT 2601 \& Department of Mathematics, The Chinese University of Hong Kong, Shatin, Hong Kong \& CNRS-ANU International Research Laboratory FAMSI (France Australia Mathematical Sciences and Interactions)}
\email{PoLam.Yung@anu.edu.au, plyung@math.cuhk.edu.hk}

\begin{document}
\begin{abstract}
We demonstrate a modification of the high-low method of Guth-Maldague-Wang \cite{GuthMaldagueWang2024} that yields a decoupling theorem which gives the sharp estimate for the discrete restriction constant for the parabola and also a diameter-free bound for the three-fold additive energy of rational points on the parabola and circle. These results were recently obtained in \cite{DengFanGuoGuoLuo2026} and \cite{CushmanDemeterWu2026}, respectively. The novel features in this argument
are an orthogonal expansion of $F^3$ and a new square function that also keeps the wavepackets discarded after pruning, resulting in new and more efficient high and low lemmas.
\end{abstract}
\maketitle
\section{Introduction}
The three-fold additive energy of a set $S$ in $\R^2$ is defined by
\[
 E_3(S)=\#\{(\xi_1,\ldots,\xi_6)\in S^6:
 \xi_1+\xi_2+\xi_3=\xi_4+\xi_5+\xi_6\}.
\]
In this paper we bound $E_3(S)$ for arbitrary finite subsets of rational points along a parabola or a circle.

\begin{theorem}
\label{thm:energy-curves}
Let $\mathcal C$ be either the rational parabola
$\{(t,t^2):t\in\mathbb Q\}$
or the rational circle
$\{(x,y)\in\mathbb Q^2:x^2+y^2=1\}.
$
Then for every finite set $S\subset \mathcal C$,
\begin{equation*}%
 E_3(S)\lesssim(\#S)^3\log(2+\#S).
\end{equation*}
\end{theorem}
Throughout this paper, for nonnegative $A$ and $B$, we write $A \lsm B$ if there is an absolute constant $C$ such that $A \leq CB$.
If $A \leq C_{s}B$ for some constant $C_s$ depending on a parameter $s$, then we write
$A \lsm_{s} B$. We write $A \sim_{s} B$ if both $A\lsm_{s} B$ and $B\lsm_{s} A$ hold.

It was observed by Bombieri and Bourgain in \cite[Theorem 1]{BombieriBourgain2015} (see also \cite[Proposition 2.15]{BourgainDemeter2015}) as a consequence of the Szemer\'edi-Trotter Theorem that $E_{3}(S) \lsm_{\vep} (\# S)^{7/2 + \vep}$. See also \cite{Mudgal2023} for work on removing this $\varepsilon$ loss and generalizing this estimate to $E_{k}(S)$ for $k \geq 3$.
In the case of the parabola, the necessity of the logarithm can be seen by taking $S = \{(n, n^2): 1 \leq n \leq M\}$ and applying the main result of Rogovskaya \cite{Rogovskaya1986} or Blomer and Br\"{u}dern \cite{BlomerBrudern2010} %
since in this case $E_{3}(S)$ is just the quadratic Vinogradov mean value.
For the circle, it has been cautiously suggested that the logarithmic factor may be unnecessary; see for example, \cite[p.~710]{KrishnapurKurlbergWigman2013}.

Recently, Cushman, Demeter, and
Wu used an interlacing method to prove the bound $E_{3}(S) \lsm_{\vep} (\# S)^{3 + \vep}$ for arbitrary finite sets
on strictly convex real curves~\cite{CushmanDemeterWu2026}, which has since been refined to obtain the first diameter-free decoupling bound for the parabola over the reals in \cite{CushmanDemeterWu2026b}. Our method does not use interlacing, and can hence work on fields such as $\Q_p$ where an ordering is absent.
We consider only rational points along curves in this paper because then we have convenient access to Fourier analysis tools over the local field $\mathbb Q_p$. %

The proof of the diameter-free bound in Theorem \ref{thm:energy-curves} proceeds via an exponential sum estimate, often referred to as discrete restriction. Write
$\mathbb T=\mathbb R/\mathbb Z$ and $e(t)=\exp(2\pi i t)$.

\begin{theorem}
\label{thm:weighted-curves}
Let the curve $\mathcal C$ and the point set $S$ be as in Theorem~\ref{thm:energy-curves}. If $D$ is
any positive integer such that $S \subset D^{-1} \mathbb Z^2$, then
\begin{equation}\label{eq:curve-weighted}
 \int_{\mathbb T^2}
 |\sum_{\xi\in S}a_\xi
 e((D\xi)\cdot x)|^6\,dx
 \lesssim\log(2+\#S)
 (\sum_{\xi\in S}|a_\xi|^2)^3
\end{equation}
for every choice of complex coefficients $(a_\xi)_{\xi\in S}$.
\end{theorem}

Theorem \ref{thm:energy-curves} follows from Theorem~\ref{thm:weighted-curves} by setting all $a_{\xi} = 1$ and noting that the left hand side of \eqref{eq:curve-weighted} reduces to $E_3(S)$.
The special set $S = \{(n, n^2): 1 \leq n \leq M\}$ has received much attention
due to connections with Strichartz estimates for the Schr\"{o}dinger equation
on the torus. Let $K(M)$ be the smallest constant such that 
\begin{align}\label{discres}
\|\sum_{n = 1}^{M}a_{n}e(nx + n^2 t)\|_{L^{6}(\T^2)} \leq K(M)(\sum_{n = 1}^{M}|a_n|^{2})^{1/2}
\end{align}
for all choices of complex numbers $\{a_n\}$. Bourgain in \cite{Bourgain1993} showed that 
$(\log M)^{1/6} \lsm K(M) \lsm \exp(O(\log M/\log\log M))$. The Bourgain-Demeter
paraboloid decoupling theorem \cite{BourgainDemeter2015} showed $K(M)\lsm_{\vep} M^{\vep}$. Subsequent improvements were obtained using the high-low method in \cite{GuthMaldagueWang2024} with
$K(M) \lsm (\log M)^{O(1)}$ and \cite{GuoLiYung2023} with $K(M) \lsm_{\vep} (\log M)^{2 + \vep}$.
The single logarithm obtained in \eqref{eq:curve-weighted} and also recently by Deng, Fan, Guo, Guo, and Luo in
\cite{DengFanGuoGuoLuo2026} matches Bourgain's initial lower bound, thus proving
a sharp estimate for $K(M)$. Bourgain in \cite{Bourgain1993} also considered the analogue of \eqref{discres} in the case of the paraboloid in $\R^{d + 1}$ for $d \geq 1$. Sharp estimates for the $d = 2$ case were recently shown by Herr and Kwak in \cite{HerrKwak2024} using methods from incidence geometry. Our proof of the sharp Strichartz estimate in the $d = 1$ case is quite different. 
It would be interesting to consider whether methods from the proof of Theorem
\ref{thm:main} below could give a different proof of the $d = 2$ case or whether
Herr and Kwak's proof can be adapted to our current setting.

Theorem \ref{thm:weighted-curves} is an immediate consequence of the Fourier decoupling theorem below (Theorem \ref{thm:main}) whose proof is based on the \text{high-low method}.
The high-low method originated in a paper by Vinh \cite{Vinh2011}, who proved a Szemer\'edi-Trotter-type theorem in finite fields. Vinh's argument relied
on two distinct contributions: a contribution from the zero frequency (the low part) and a contribution from the nonzero frequencies (the high part). This method was subsequently brought into harmonic analysis
by Guth, Solomon, and Wang in \cite{GuthSolomonWang2019} and Guth, Wang, and Zhang in \cite{GuthWangZhang2020}. Since then
the high-low method has played an important role in harmonic analysis, incidence geometry, and combinatorics. See \cite{CohenPohoataZakharov2023, GanGuoGuthHarrisMaldagueWang2022, GanMaldagueOh2025, GuthMaldague2026, Maldague2022, RenWang2023} for a sample of numerous other developments enabled by the high-low method.

The high-low method was first introduced into Fourier decoupling by Guth, Maldague, and Wang in \cite{GuthMaldagueWang2024}.
The method was later adapted to work over $\Q_p$ in \cite{GuoLiYung2023} by Guo and the two authors of this manuscript; see also \cite{LiLiYung2023} for some heuristics. 
We refer the reader to \cite{GuoLiYung2023,Li2024} for the relevant aspects of Fourier analysis on $\mathbb Q_p$. In particular, we normalize Haar measure $dx$ on $\Qp^2$ by $|\cO^2|=1$, and fix an additive
character $\chi \colon \mathbb Q_p \to \mathbb C^{\times}$ that is trivial ($\equiv 1$) on $\cO$ but not trivial on
$p^{-1}\cO$. The Fourier transform is then defined by
\[
 \widehat f(\xi)=\int_{\mathbb Q_p^2}
 f(x)\overline{\chi(x\cdot\xi)}\,dx.
\]
See \cite[Section 2]{GuoLiYung2023} for details
on tools such as the locally constant property, uncertainty principle, and wavepacket decomposition
over $\Q_p$ that we will often make use of here without proof.

\subsection{Setup and statement of main theorem (Theorem \ref{thm:main})}

\begin{definition}\label{def:good-arc}
Let $p$ be an odd prime with $p \equiv 2 \pmod{3}$.
A normalized parabolic arc is a graph
\[
 \Gamma=\{(t,\phi(t)):t\in\cO\},
\]
where $\phi:\cO\to\cO$ is given by a convergent power series
\begin{equation} \label{eq:powerseries}
\phi(t) = \sum_{n=2}^{\infty} c_n t^n \quad \text{for all $t \in \cO$},
\end{equation}
with $c_n \in \cO$ for all $n$, and $\kappa := c_2$ satisfies $|\kappa|_p = 1$ while $|c_n|_p < 1$ for all $n \geq 3$.
Note that the power series of $\phi(t)$ in \eqref{eq:powerseries} is convergent for all $t \in \cO$ if and only if $|c_n|_p \to 0$ as $n \to \infty$.
\end{definition}

Condition~\eqref{eq:powerseries} with the conditions on the coefficients $c_n$ says that the reduction of 
the graph modulo $p$ is a quadratic graph with nonzero quadratic
coefficient, that is,
\begin{equation}\label{eq:normal-form}
 \phi(t)\equiv\kappa t^2\pmod p
\end{equation}
coefficientwise.
The parabola is the basic example, with
$\phi(t)=t^2$. For the circle, one would like to consider the function $\sqrt{1-t^2}$, but this function only admits a convergent power series expansion for $|t|_p \leq p^{-1}$ (since if $p$ is an odd prime, any $p$-adic integer that is $\equiv 1 \pmod p$ has a square root that is also $\equiv 1 \pmod p$ and defined by a convergent power series). Thus one should consider instead 
\begin{equation}\label{eq:circlephi}
    \phi(t) = p^{-2} (\sqrt{1-(p t)^2} - 1),
\end{equation} which has 
a convergent power series 
$-\frac{1}{2} t^2 + \sum_{n \geq 2} \binom{1/2}{n} (-1)^n p^{2n-2} t^{2n}$
on $\cO$. In this case $\Gamma = \{(t,\phi(t)) \colon t \in \cO\}$ is the set of all
points $(X,Y) \in \Z_p^2$ that lie on the ellipse $(pX)^2 + (1 + p^2 Y)^2 = 1$. 
When $p \equiv 3 \pmod 4$, i.e. when $-1$ is not a quadratic residue mod $p$, every rational point on the unit circle has $p$-adic integer coordinates. In that case, let $\mathcal{C}_{0,1}$ be the set of rational points $(x,y) \in \Q^2$ on the unit circle $x^2+y^2=1$ with $x \equiv 0 \pmod p$ and $y \equiv 1 \pmod p$ when considered as elements of $\Z_p$. For $(x,y) \in \mathcal{C}_{0,1}$, one could write $x = pX$ for some $X \in \Z_p$, which forces (via the equation of the circle) that $y = 1 + p^2 Y$ for some $Y \in \Z_p$, and under the invertible affine change of variables $(x,y) \mapsto (X,Y)$, the set $\mathcal{C}_{0,1}$ gets carried into the ellipse $\Gamma$. An extension of this argument shows that the rational circle can be divided into a finite number
of pieces, each of which becomes part of a normalized parabolic arc after an invertible affine change of coordinates; see Section~\ref{sect:circle}. 

It will be important that the class of normalized parabolic arcs is invariant under parabolic rescaling; see Proposition~\ref{prop:rescalinggeom}.

Let $N \geq 0$ and let $F  \colon \Q_p^2 \to \mathbb C$ be Fourier supported in the $p^{-2N}$ neighborhood of a normalized parabolic arc $\Gamma$ in $\Q_p^2$. In other words, the Fourier support of $F$ is contained in 
\[
 \{(\xi,\eta):\xi\in\cO,
                  |\eta-\phi(\xi)|_p\leq p^{-2N}\}.
\]

For $k \geq 0$, let $\I_k$ be a partition of $\cO$ into $p$-adic intervals of length $p^{-k}$. For $I \in \I_k$, write $F_I$ for the Fourier localization of $F$ to $I \times \Q_p$:
\begin{align*}%
 \widehat{F_I}(\xi,\eta)=\1_I(\xi)\widehat F(\xi,\eta).
\end{align*}
We write $K$ for the terminal intervals in $\I_N$ so that $F = \sum_{K \in \I_N} F_K$. The main decoupling theorem in this paper is as follows.

\begin{theorem}\label{thm:main}
Let $p$ be an odd prime with $p\equiv2\pmod3$. Then there is a
constant $C_p$, uniform over normalized parabolic arcs $\Gamma$, with the
following property. Suppose for some nonnegative integer $N$,
an $L^2$ function $F \colon \Q_p^2 \to \mathbb C$ is Fourier supported in the $p^{-2N}$ neighborhood of a normalized parabolic arc $\Gamma$. Assume additionally that $F_K$ is not identically zero for at most $M$ terminal caps $K \in\I_N$. Then
\begin{equation}\label{eq:main}
 \int_{\mathbb Q_p^2}|F|^6
 \leq C_p\log(2+M)
 (\sum_{K\in\I_N}\norm{F_K}_\infty^2)^2
 (\sum_{K\in\I_N}\norm{F_K}_2^2).
\end{equation}
\end{theorem}
A reader interested only in the case of Theorem \ref{thm:main} for the parabola may take $p = 5$ and $\phi(t) = t^2$ throughout.

It is easy to extend Theorem~\ref{thm:main} to affine images of normalized parabolic arcs.
This will be necessary when we apply Theorem~\ref{thm:main} to subsets of the rational parabola whose coordinates may not be $p$-adic integers, and to ellipses that arise naturally by rescaling circles.

Theorem~\ref{thm:main} will be applied with $p = 5$ for the parabola and $p = 11$ for all rational circles to give Theorem~\ref{thm:weighted-curves}. In fact,
\begin{equation} \label{eq:torustoQp}
 \int_{\mathbb T^2}
 |\sum_{\xi\in S}a_\xi
 e((D\xi)\cdot x)|^6\,dx
 = \frac{1}{|B_R|} \int_{\Q_p^2} | \sum_{\xi \in S} a_{\xi} \chi(\xi \cdot x) 1_{B_R}(x) |^6\, dx
\end{equation}
where $B_R$ is a sufficiently large ball of radius $R$ in $\Q_p^2$ centered at the origin. It suffices to choose $R = p^{2N}$ with $N$ large enough so that $|\xi - \xi'|_p > p^{-N}$ for every distinct pair $\xi, \xi' \in S$, and so that
\[
p^{2N} > \max_{\substack{(\xi_1,\dots,\xi_6) \in S^6 \\ \xi_1+\xi_2+\xi_3-\xi_4-\xi_5-\xi_6 \ne 0}} |\xi_1+\xi_2+\xi_3-\xi_4-\xi_5-\xi_6|_p^{-1},
\]
because then for any $(\xi_1, \dots, \xi_6) \in S^6$, 
\[
\xi_1+\xi_2+\xi_3=\xi_4+\xi_5+\xi_6 \Longleftrightarrow \xi_1+\xi_2+\xi_3 \equiv \xi_4+\xi_5+\xi_6 \pmod{p^{2N}}.
\]
The right hand integral in \eqref{eq:torustoQp} can be bounded by Theorem~\ref{thm:main}. 
Note that the bound on the right hand side of \eqref{eq:main} depends only on the number $M$ of nonzero terminal pieces. This improved bound is what makes it possible to establish the diameter-free bounds in Theorem~\ref{thm:energy-curves}.

It seems plausible that with non-trivial additional effort, the methods for proving Theorem \ref{thm:main} will also work over $\R$, with possibly additional log losses arising from wavepacket decompositions over $\R$. This may in turn allow one to establish a non-sharp variant of Theorem \ref{thm:energy-curves} for arbitrary finite subsets along convex real curves.

\subsection{Sketch of proof} Below we sketch the proof of Theorem~\ref{thm:main}, highlighting new ideas not present in the existing literature.

Recall that we denote the terminal frequency caps by $K$. For $0 \leq k < N$, each $I \in \I_k$ has $p$ immediate children. The set of immediate children of $I$ is denoted $C(I)$, and elements of $C(I)$ will usually be denoted $J$. We also write $\rho_k$ for $p^{-k}$, the common length of the intervals in $\I_k$. The (finite) residue field $\Z_p / p \Z_p$ will be denoted $\mathbb{F}_p$.

\subsubsection{Orthogonal expansion of $F^3$} \label{sec:cubic}

For $I\in\I_k$, $k<N$, define
\begin{align*}%
 T_I(F)=F_I^3-\sum_{J\in C(I)}F_J^3.
\end{align*}
Successively expanding each $F_I^3$ gives the algebraic identity
\begin{equation*} 
 F^3=\sum_{k=0}^{N-1}\sum_{I\in\I_k}T_I(F)
       +\sum_{K\in\I_N}F_K^3.
\end{equation*}
It turns out that since $\widehat F$ is supported near a normalized parabolic arc, when $p$ is an odd prime with $p \equiv 2 \pmod 3$, the pieces on the right have disjoint Fourier supports. In fact, $p \equiv 2 \pmod 3$ is only used to derive the identity 
\[
 |r^2+rs+s^2|_p=\max(|r|_p,|s|_p)^2,
\]
which plays a crucial role in proving that the $T_I(F)$ have disjoint Fourier supports as $I$ varies over $\bigcup_{k=0}^{N-1} \I_k$. Hence $L^2$
orthogonality gives
\begin{equation}\label{eq:intro-cubic-tree}
 \norm F_6^6
 =\sum_{k=0}^{N-1}\sum_{I\in\I_k}\norm{T_I(F)}_2^2
  +\sum_{K\in\I_N}\norm{F_K}_6^6,
\end{equation}
from which we control the left hand side of \eqref{eq:main}. Note that 
\[
T_I(F) = \Big(\sum_{J \in C(I)} F_J\Big)^3 - \sum_{J \in C(I)} F_J^3.
\]
Denote by $C_3(I)$ the set of ordered tuples $(J_1,J_2,J_3) \in C(I)^3$ where $J_1, J_2, J_3$ are not all equal. 
As a result,
\[
T_I(F) = \sum_{(J_1,J_2,J_3) \in C_3(I)} F_{J_1} F_{J_2} F_{J_3}.
\]

\subsubsection{Induction on scales via normalized pieces}

To prove Theorem~\ref{thm:main}, after rescaling $F$, we may assume
\begin{equation}\label{eq:intro-normalization}
 \sum_{K\in\I_N}\norm{F_K}_\infty^2=1.
\end{equation}
In this normalization, it suffices to prove
\begin{equation}\label{eq:intro-normalized-goal}
 \norm F_6^6\lesssim_p m\norm F_2^2,
\end{equation}
where at most $2^m$ terminal pieces are nonzero ($m \geq 1$). To do so we induct on scale.
For every frequency interval $I$, set
\[
 A_I=\sum_{\substack{K\in\I_N\\K\subset I}}
       \norm{F_K}_\infty^2,
 \qquad E_I=\norm{F_I}_2^2.
\]
When $A_I>0$, we normalize $F_I$ by defining $f_I=A_I^{-1/2}F_I.$
Thus $f_I$ obeys the normalization \eqref{eq:intro-normalization}
on the frequency interval $I$. 
For 
$\mathbf J=(J_1,J_2,J_3)\in C_3(I)$, the broad set is defined using
the normalized functions $f_{J_1},f_{J_2},f_{J_3}$ by
\[
 \mathcal B_{\mathbf J}
 =\{x:\sum_{i=1}^3|f_{J_i}(x)|^6 \leq 6p^3|f_{J_1}(x)f_{J_2}(x)f_{J_3}(x)|^2\}.
\]
As will be shown in Proposition~\ref{prop:broad} below, the main estimate for one normalized triple is
\begin{equation}\label{eq:intro-broad-triple}
 \int_{\mathcal B_{\mathbf J}}
 |f_{J_1}f_{J_2}f_{J_3}|^2
 \lesssim_p m\sum_{i=1}^3\norm{f_{J_i}}_2^2.
\end{equation}
Our induction hypothesis bounds this integral on the complement of the broad set by the same bound. Consequently,
\begin{equation}\label{eq:intro-full-triple}
 \int_{\Qp^2}|f_{J_1}f_{J_2}f_{J_3}|^2
 \lesssim_p m\sum_{i=1}^3\norm{f_{J_i}}_2^2,
\end{equation}
which feeds into \eqref{eq:intro-cubic-tree} via
\begin{equation} \label{eq:TIF_intro_bdd}
\|T_I(F)\|_2^2 \lesssim_p \sum_{(J_1,J_2,J_3) \in C_3(I)} \int_{\Qp^2} |F_{J_1} F_{J_2} F_{J_3}|^2 = \sum_{(J_1,J_2,J_3) \in C_3(I)} A_{J_1} A_{J_2} A_{J_3} \int_{\Qp^2} |f_{J_1} f_{J_2} f_{J_3}|^2.
\end{equation}

\subsubsection{The telescoping estimate} %

In fact, combining \eqref{eq:TIF_intro_bdd} with  \eqref{eq:intro-full-triple} and using $\|f_J\|_2^2 = A_J^{-1} E_J$ gives the key bound
\begin{equation} \label{eq:intro-TI-bdd}
 \norm{T_I(F)}_2^2
 \lesssim_p m
 (A_I^2E_I-\sum_{J\in C(I)}A_J^2E_J)
\end{equation}
which one can substitute back into \eqref{eq:intro-cubic-tree}.
We sum \eqref{eq:intro-TI-bdd} over all non-terminal frequency caps $I$, and the right hand side telescopes:
\begin{equation*}%
 \sum_{k=0}^{N-1}\sum_{I\in\I_k} (A_I^2E_I-\sum_{J\in C(I)}A_J^2E_J)= \norm F_2^2-
 \sum_{K\in\I_N}\norm{F_K}_\infty^4
                         \norm{F_K}_2^2.
\end{equation*}
The terminal term in \eqref{eq:intro-cubic-tree} can be estimated via
$\norm{F_K}_6^6\leq\norm{F_K}_\infty^4\norm{F_K}_2^2.$
Hence $$\|F\|_6^6 \lesssim_p m (\|F\|_2^2 - \sum_{K \in \I_N} \|F_K\|_{\infty}^4 \|F_K\|_2^2 ) + \sum_{K \in \I_N} \|F_K\|_{\infty}^4 \|F_K\|_2^2 \lesssim_p m\|F\|_2^2,$$ establishing \eqref{eq:intro-normalized-goal} modulo the proof of \eqref{eq:intro-broad-triple}.
Observe that this telescoping mechanism allows us to sum over all scales $k$ without incurring any log loss.

\subsubsection{Amplitude decomposition}

To prove \eqref{eq:intro-broad-triple}, it is not too hard to see that the contribution from the region where
$|f_{J_1}f_{J_2}f_{J_3}|\lsm 1$ can be bounded directly. On the other hand, for dyadic
$\alpha\gtrsim 1$, set
\[
 U_\alpha=\{x\in\mathcal B_{\mathbf J}:
 \alpha^3<|f_{J_1}(x)f_{J_2}(x)f_{J_3}(x)|\leq(2\alpha)^3\}.
\]
Since $|f_{J_i}|\leq2^{m/2}$, only $O(m)$ of these sets are nonempty.
Thus \eqref{eq:intro-broad-triple} follows from the uniform estimate
\begin{equation}\label{eq:intro-level-set}
 \alpha^6|U_\alpha|\lesssim_p
 \sum_{i=1}^3\norm{f_{J_i}}_2^2.
\end{equation}
Summing this estimate over the $O(m)$ values of $\alpha$ for which $U_{\alpha}$ is nonempty is
the only source of the logarithmic loss in the proof of Theorem~\ref{thm:main}.

\subsubsection{A compensated square function and a new low lemma}

We now prove \eqref{eq:intro-level-set} using a modified version of the high-low method, following Guth, Maldague, and Wang \cite{GuthMaldagueWang2024}. Fix $\alpha$ and put $f=\sum_{J \in \{J_1,J_2,J_3\}} f_{J}$. Starting from
$f_{N+1}=f$, construct $f_N,\ldots,f_1$ from fine to coarse, using a
pruning height $L$ chosen so that $\alpha L\sim_p 1$. Once
$f_{k+1}$ has been constructed, decompose it into its frequency pieces
$f_{k+1,I}$, $I\in\I_k$, prune by removing wavepackets with modulus above $L$, and denote
the sum of the retained pieces by $f_k$. We write $f_{k,I}$ for the
retained part of $f_{k+1,I}$; thus $|f_{k,I}|\leq L$, while
$f_{k+1,I}-f_{k,I}$ is the discarded part. The usual square function
at scale $k$ would be
\[
 \sum_{I\in\I_k}|f_{k+1,I}|^2.
\]
We instead use
\begin{equation} \label{eq:intro-sq-def}
 g_k:=\sum_{I\in\I_k}|f_{k+1,I}|^2
 +\sum_{j=k+1}^N\sum_{I\in\I_j}
   |f_{j+1,I}-f_{j,I}|^2.
\end{equation}
Keeping track of the discarded pieces in $g_k$ leads to a particularly clean bound for the difference between the original function and the pruned pieces:
\begin{equation} \label{eq:intro-prune-error}
|f_J - f_{k+1,J}| \leq \frac{1}{L} g_{k+1}, \quad J \in \I_{k+1}
\end{equation}
(as opposed to $\sum_{j \geq k+1} g_j$ which we could only bound by trivially summing over $j$ and losing a log). Since the correction terms are all low-frequency, this leads to a new low lemma (Lemma~\ref{lem:low}) that says
\begin{equation*} %
 P_{k+1}g_k=g_{k+1},
\end{equation*}
where $P_{k+1}$ is the Fourier projection onto the ball $B_{p^{-(k + 1)}}$ of radius $p^{-(k+1)}$ centered at the origin.
Note we have an equality (as opposed to an inequality). Thus the high-frequency part is exactly the difference of two
successive compensated square functions:
\begin{equation*}%
 h_k=(1-P_{k+1})g_k
     =g_k-g_{k+1}.
\end{equation*}
The function $h_k$ is commonly denoted by $g_{k,\mathrm{high}}$ in the
high-low literature. In particular, sums of the high parts telescope:
\begin{equation}\label{eq:intro-high-telescope}
 g_k=g_N+\sum_{j=k}^{N-1}h_j.
\end{equation}

\subsubsection{A new high lemma exploiting orthogonality between all scales}

Motivated by \eqref{eq:intro-high-telescope}, consider the sum of all the high parts
\[
 H:=\sum_{k=1}^{N-1}h_k.
\]
Our new high lemma (Lemma~\ref{lem:high}) reads
\begin{equation} \label{eq:intro-Hnorm}
\|H\|_2^2 = \sum_{k=1}^{N-1}\norm{h_k}_2^2
 \lesssim_p L^2\sum_{i=1}^3\norm{f_{J_i}}_2^2.
\end{equation}
The key point is that the $h_k$ are orthogonal across scales, so \eqref{eq:intro-Hnorm} incurs no loss in the number of scales.

\subsubsection{A new high-low dichotomy}
We now run the high-low machinery to prove
\eqref{eq:intro-level-set}. In previous generations of the high-low argument, one typically determined whether a point $x$ belonged to a high set $\Omega_k$ by comparing $g_k(x)$ to its high part $h_k(x)$: a point belonged to the $k$-th high set if $g_k(x) \leq N |h_k(x)|$. The log factor $N \sim_p \log(p^N)$ was there to ensure a pointwise estimate 
\begin{equation*} %
g_{k+j}(x) \leq (1 + O(1/N)) g_{k+j+1}(x) \quad \text{on $ \Q_p^2 \setminus \Omega_{k+j}$, \quad for $j = 1, \dots, N-k-1$},
\end{equation*}
which could be composed together since $(1+O(1/N))^{N-k} \lesssim 1$ uniformly in $N$.
The new insight here is to allow for a multiscale pointwise estimate on the high set, which avoids the aforementioned composition. In particular, one can allow 
\begin{equation} \label{eq:intro-new-high-def1}
g_k(x) \lesssim |\sum_{j=k}^{N-1} h_j(x) |
\end{equation}
for $x \in \Om_k$ where one permits some low frequencies of $g_k$ to remain present on the right hand side. This is sufficient to close the argument without unnecessary logs thanks to the global high lemma \eqref{eq:intro-Hnorm}. 

Because of our new low lemma, the right hand side of \eqref{eq:intro-new-high-def1} is the modulus of $g_k-g_N$ (see  \eqref{eq:intro-high-telescope}). This motivates one to define our new high sets $\Omega_k$ for $1 \leq k < N$ by
\begin{equation} \label{eq:intro-new-high-def2}
 \Omega_k=\{g_k>2g_N\}\cap
 \bigcap_{j=k+1}^{N-1}\{g_j\leq2g_N\}.
\end{equation}
This implies $g_k - g_N > g_k/2$ on $\Omega_k$, which shows \eqref{eq:intro-new-high-def1} with implicit constant 2.
The key is to bound $\sum_{k=1}^{N-1} \alpha^6 |U_{\alpha} \cap \Omega_k|$.
The definition of $U_{\alpha}$ and the pointwise pruning estimate \eqref{eq:intro-prune-error} give
\[
 \alpha^6\lesssim_p\alpha^2
 |f_{k+1,J_1}f_{k+1,J_2}|^2
\]
on $U_\alpha\cap\Omega_k$.
Local bilinear restriction then yields
\begin{equation} \label{eq:intro-bilinear}
 \alpha^6|U_\alpha\cap\Omega_k|
 \lesssim_p\alpha^2\int_{\Omega_k}g_k^2.
\end{equation}
Moreover, on $\Omega_k$, \eqref{eq:intro-new-high-def1} implies
\begin{equation}\label{eq:intro-high-domination}
 g_k\leq 2|P_{k}H|,
\end{equation}
where $P_{k}H$ is the low-frequency part of $H$.

The locally constant property shows that $g_k, \ldots, g_N$ are constant on every square of side length $p^k$
and hence $\Om_k$ is the disjoint union of such squares.
Next, note that $(P_{k}H)(x) = (\wc{1}_{B_{p^{-k}}} \ast H)(x) = p^{-2k}(1_{B_{p^k}} \ast H)(x) = \frac{1}{|Q_x|}\int_{Q_x}H(y)\, dy$
where $Q_x$ is the square of side length $p^k$ centered at $x$. For each $x$, either $Q_x$ is completely contained
in $\Omega_k$ or entirely disjoint from it. Thus
\begin{align}\label{proj}
P_{k}(1_{\Om_k}H)(x) = \frac{1}{|Q_x|}\int_{Q_x}1_{\Om_k}(y)H(y)\, dy = \frac{1}{|Q_x|}\int_{Q_x}1_{\Om_k}(x)H(y)\, dy = 1_{\Om_k}P_{k}H(x).
\end{align}
Combining this with \eqref{eq:intro-high-domination} gives
\begin{equation} \label{eq:intro-sumgk}
 \sum_{k=1}^{N-1}\int_{\Omega_k}g_k^2 \leq \sum_{k=1}^{N-1} \int_{\Omega_k} 4 |H|^2 
 \leq4\norm H_2^2.
\end{equation}
Note that the sum over $k$ does not incur any extra loss here, because we exploited the disjointness of $\Omega_k$. Consequently, using \eqref{eq:intro-bilinear}, \eqref{eq:intro-sumgk} and \eqref{eq:intro-Hnorm},
\begin{equation*} %
 \sum_{k=1}^{N-1}\alpha^6|U_\alpha\cap\Omega_k|
 \lesssim_p\alpha^2L^2
 \sum_{i=1}^3\norm{f_{J_i}}_2^2.
\end{equation*}
This is $\lesssim_p \sum_{i=1}^3 \|f_{J_i}\|_2^2$ since
$\alpha L\sim_p1$. Together with the simpler estimate on the low set, this establishes \eqref{eq:intro-level-set}.

\subsection{The rational circle} \label{sect:circle}

Here we explain in more detail how we handle the rational circle. 

\begin{lemma}\label{lem:aniso}
Let $p$ be a prime with $p\equiv3\pmod4$. Then for all $x,y\in\Qp$,
\[
  |x^2+y^2|_p=\max(|x|_p,|y|_p)^2 .
\]
\end{lemma}

\begin{proof}
Both sides vanish when $x=y=0$, and both are symmetric in $x$ and $y$,
so we may assume $y\neq0$ and $|x|_p\leq|y|_p$. Put $t:=x/y$, so
$t\in\cO$ and $x^2+y^2=y^2 (t^2+1).$
Since $p\equiv3\pmod4$, Euler's criterion gives
$(-1)^{(p-1)/2}\equiv -1\pmod{p}$, so $-1$ is not a square in $\mathbb F_p$. Hence
the reduction $\bar t^{\,2}+1$ is a nonzero element of $\mathbb F_p$,
which says exactly that $|t^2+1|_p=1$. Therefore
\[
  |x^2+y^2|_p=|y|_p^2=\max(|x|_p,|y|_p)^2 . \qedhere
\]
\end{proof}

\begin{lemma}%
Let $p$ be a prime with $p\equiv3\pmod4$ and let $(x,y)\in\Qp^2$
satisfy $x^2+y^2=1$. Then $x,y\in\cO$. In particular every rational
point of the unit circle lies in $\cO^2$.
\end{lemma}

\begin{proof}
By Lemma~\ref{lem:aniso},
$\max(|x|_p,|y|_p)^2=|x^2+y^2|_p=|1|_p=1,$
so $\max(|x|_p,|y|_p)=1$ and both coordinates lie in $\cO$. 
\end{proof}

Fix now a prime $p$ with both $p \equiv 2 \pmod 3$ and $p \equiv 3 \pmod 4$ ($p = 11$ will do). Let $\mathcal Z$ be the set of solutions 
$\{(a,b) \in \mathbb{F}_p^2 \colon a^2 + b^2 = 1\}.$
Then the rational circle is contained in a finite union
$\bigcup_{(a,b) \in \mathcal Z} \mathcal{C}_{a,b}$
where each 
\[
\mathcal{C}_{a,b} := \{(x,y) \in \Z_p^2 \colon x^2 + y^2 = 1, \quad x \equiv a \pmod{p}, \quad y \equiv b \pmod{p}\}
\]
is an affine image of a normalized parabolic arc: one obtains $\mathcal{C}_{0,1}$ from $\mathcal{C}_{a,b}$ by a rotation $(x,y) \mapsto (B x - A y, A x + B y)$ where $A, B \in \Z_p$ are liftings of $a$ and $b$ so that $A^2 + B^2 = 1$ in $\Z_p$, and $a \equiv A \pmod p$, $b \equiv B \pmod p$ (the existence of $A$ and $B$ is guaranteed by Hensel's lemma since $p$ is odd). Furthermore, $\mathcal{C}_{0,1}$ is an affine image of the ellipse parametrized by $\{(t,\phi(t)) \colon t \in \cO\}$ where $\phi(t)$ was given by \eqref{eq:circlephi}. If $S$ is any finite subset of the rational circle, we partition
\[
S = \bigcup_{(a,b) \in \mathcal Z} S_{a,b}, \quad S_{a,b}:= S \cap \mathcal{C}_{a,b}
\]
and bound the exponential sum moment
\[
\int_{\T^2} | \sum_{\xi \in S} a_{\xi} e((D\xi) \cdot (x,y)) |^6 dx dy \lesssim_p \sum_{(a,b) \in \mathcal Z} \int_{\T^2} | \sum_{\xi \in S_{a,b}} a_{\xi} e((D\xi) \cdot (x,y)) |^6 dx dy.
\]
Since each $\mathcal{C}_{a,b}$ is an affine image of a normalized parabolic arc, Theorem~\ref{thm:main} applies. We apply the reduction given immediately after Theorem~\ref{thm:main}, and obtain Theorem~\ref{thm:weighted-curves} for the circle as well.

\subsection{Use of AI tools}
This project began as a series of notes 
that the authors wrote when the first author visited the
second author at the Australian National University
in July 2023.
At that time, we were trying to work out the number of log powers in subsequent versions of the high-low method
\cite{GuthMaldague2022, Johnsrude2025} that came
after \cite{GuthMaldagueWang2024}. Despite many optimizations along these lines, our best result was only $K(M) \lsm (\log M)^{10/6}$ which we did not publish.

On July 21, 2026, the first
author decided to upload these notes to GPT-5.6 Sol
and asked whether the number of log powers could be reduced.
GPT-5.6 Sol gave the first key observation of the orthogonal expansion of $F^3$ (Proposition \ref{orthogonalexpansion}) and was able to show
that $K(M) \lsm (\log M)$.
Asking whether this bound could be improved
gave $K(M) \lsm (\log M)^{1/2}$ which resulted from creating the absolute stopping sets \eqref{eq:intro-new-high-def2} (as opposed to the constraint that $g_k \lsa g_{k}^h$ found in older works) along with a prototype of what became the new High Lemma, Lemma \ref{lem:high}.
Further prompting gave no further improvements.
On September 6, 2026, with the release of GPT-6 Astra, the first author again asked whether an improvement could be found, but this yielded no results. 
Asking if a particular dyadic pigeonholing could be avoided suddenly yielded the new compensated square functions \eqref{eq:intro-sq-def} and a more efficient pruning approximation \eqref{eq:intro-prune-error} which gave the sharp upper bound of $K(M) \lsm (\log M)^{1/6}$.
Over the next few weeks, the authors digested and simplified the AI-generated manuscript, yielding the argument we provide below.
GPT-6 Astra was also used to create the \texttt{tikz} code for Figure \ref{tiffig} and proofread the final manuscript.

\subsection{Comparison with \cite{DengFanGuoGuoLuo2026}}
On September 16, 2026, while preparing this manuscript, the authors were made aware of \cite{DengFanGuoGuoLuo2026} which also independently proved the sharp discrete restriction estimate for the parabola using essentially the same methods. We briefly comment on some similarities and differences between our paper and theirs.

Equations (38) and (39) of \cite{DengFanGuoGuoLuo2026} correspond to 
the definitions of $f_{k, I}$ and $e_{k, I}$ in \eqref{fkidef} and \eqref{ekidef}, respectively. The definition of $G_{k}(z)$ on page 11 of \cite{DengFanGuoGuoLuo2026} is exactly the compensated square function in \eqref{corrsquare}. The biggest difference
between the argument here and the one in \cite{DengFanGuoGuoLuo2026} is that we view $\nms{F}_{L^{6}}^{6} = \nms{F^3}_{L^{2}}^{2}$ while
they use that it is equal to $\nms{|F|^2}_{L^{3}}^{3}$. The former allows us to make
use of a clever $L^2$ orthogonal expansion of $F^3$ (Proposition \ref{orthogonalexpansion}) while they need: 
\begin{enumerate}[(a)]
\item a (weak) $L^3$ ``Littlewood-Paley" type theorem to separate the contributions to $|F|^2$ from different frequency scales (see \cite[Lemma 3.1]{DengFanGuoGuoLuo2026});
\item an additional interpolation, between an $L^2$ orthogonal estimate for the contributions to $|F|^2$ from the same frequency scale and a trivial $L^{\infty}$ estimate, to obtain control on some $\ell^{3/2}$ sums of (weak) $L^3$ spatial norms of these different contributions (see \cite[Lemma 3.4]{DengFanGuoGuoLuo2026}).
\end{enumerate}
The arguments in our paper are slightly simpler in that regard, but it is not immediately clear how one could use our method to obtain a weak $L^6$ discrete Strichartz estimate as in \cite{DengFanGuoGuoLuo2026}. Our arguments are also presented in the language of the high-low method in decoupling over $\Q_{p}$ rather than exponential sums, and clarify the role of $(\sum_n |a_n|^2)^3$ in \cite{DengFanGuoGuoLuo2026}: one could take two copies of $\sum_n |a_n|^2$ as $\sum_K \|F_K\|_{\infty}^2$, and one copy as a normalized version of $\sum_K \|F_K\|_2^2$.

\subsection{Acknowledgments}
Li is partially supported by DMS-2453448. Yung is partially supported by a
Discovery Project DP250103744 from the Australian Research Council. We also acknowledge kind support from the American Institute of Mathematics through the Fourier restriction research community.

\section{Geometry of normalized parabolic arcs}

\begin{lemma}[Taylor expansion]\label{lem:taylor}
Let $\phi$ be as in Definition~\ref{def:good-arc}. For $a\in\cO$ set
\[
  b_n(a):=\sum_{m\geq n}\binom{m}{n}c_m\,a^{m-n}
\]
where $c_0, c_1$ are understood to be 0.
Then $b_n(a)$ converges in $\cO$, $b_1(a) = \phi'(a)$, $b_2(a)\equiv\kappa \pmod p$, $b_n(a)\in p\cO$ for all $n\geq3$, and $$\phi(a+h)=\phi(a) + \sum_{n = 1}^{\infty} b_n(a)h^n$$ for all
$h\in\cO$.
\end{lemma}

\begin{proof}
Let $a, h \in \cO$. Since
\[
\Big| \binom{m}{n} c_m a^{m-n} h^n \Big|_p \leq |c_m|_p \to 0 \quad \text{as $m \to \infty$},
\]
we see that $|b_n(a)|_p \leq \sup_{m \geq n} |c_m|_p \leq 1$, so $b_n(a) \in \cO$ for all $n$, and $|b_n(a)|_p \to 0$ as $n \to \infty$. From $|c_m|_p < 1$ for all $m \geq 3$, one also obtains $b_n(a) \in p \cO$ for all $n \geq 3$, and $b_2(a) \equiv c_2 = \kappa \pmod p$ as desired. It is also easy to see $b_1(a) = \sum_{m \geq 1} m c_m a^{m-1} = \phi'(a)$. Finally 
we expand
\[
\phi(a+h) = \sum_{m = 2}^{\infty} c_m (a+h)^m = \sum_{m = 2}^{\infty} \sum_{n=0}^m \binom{m}{n} c_m a^{m-n} h^n.
\]
Rearranging the series and noting $b_0(a) = \phi(a)$ gives the desired expansion for $\phi(a+h)$. 
\end{proof}

\begin{remark}
 In proving the above lemma it may be tempting to just write 
\[
\phi(a+h) = \sum_{n=0}^{\infty} \frac{\phi^{(n)}(a)}{n!} h^n
\]
but then it is not as easy to justify that $\phi^{(n)}(a)/n! \in \cO$ since even though $\phi^{(n)}(a) \in \cO$ for all $n$, in general $1/n!$ is not in $\cO$.
\end{remark}

For $k \geq 0$ and $I \in \I_k$, let 
\begin{align*}%
\Xi_I := \{(\xi,\eta) \in \cO^2 \colon \xi \in I, |\eta - \phi(\xi)|_p \leq p^{-\min\{2k,2N\}}\},
\end{align*}
so that the Fourier support of $F_I$ is contained in $\Xi_I$.
Then Lemma~\ref{lem:taylor} shows that 
\[
\Xi_I = \{(\xi,\eta) \colon |\xi-a|_p \leq p^{-k}, |\eta - \phi(a) - \phi'(a) (\xi-a)|_p \leq p^{-\min\{2k,2N\}}\}
\]
whenever $a$ is a point in $I$, so $\Xi_I$ is a parallelepiped, and $|F_I|$ is constant on translates of the dual parallelepiped
\begin{align*}%
T_{I}^{0} := \{(x,y) \in \Q_p^2 \colon |x + \phi'(a) y |_p \leq p^k, |y|_p \leq p^{\min\{2k,2N\}} \}.
\end{align*}
We will denote the set of all such translates by $\T_I$.

Proposition \ref{prop:rescalinggeom} shows that the class of normalized parabolic arcs is invariant under parabolic scaling.

\begin{proposition}\label{prop:rescalinggeom}
Let $\Gamma = \{(t,\phi(t)) \colon t \in \cO\}$ be a normalized parabolic arc, and  $F$ be Fourier supported in a $p^{-2N}$ neighborhood of $\Gamma$. 
\begin{enumerate}[(a)]
\item If $k \geq 0$ and $I \in \I_k$, let $\Psi_I$ be the invertible affine map given by
\[
\Psi_I(\xi,\eta) := (a,\phi(a)) + L_I(\xi,\eta), \quad L_I(\xi,\eta):= (p^k \xi, \phi'(a) p^k \xi + p^{2k} \eta),
\]
where $a$ is any point in $I$.
Then $\Gamma_I := \Psi_I^{-1} (\Gamma|_I)$ is also a normalized parabolic arc, which can be written as $\{(t,\phi_I(t)) \colon t \in \cO\}$ where
\[
\phi_I(t) := p^{-2k} ( \phi(a+p^k t) - \phi(a) - \phi'(a) p^k t ).
\] 
\item If $0 \leq k \leq N$ and $I \in \I_k$, then
\begin{align*}
    \widetilde{F_I}(x,y) := F_I(L_I^{-t}(x,y)) \chi(-(a,\phi(a)) \cdot L_I^{-t}(x,y)) 
\end{align*}
is Fourier supported in a $p^{2(k-N)}$ neighborhood of $\Gamma_I$.
\end{enumerate}
\end{proposition}

 \begin{proof}
For $(a)$, by Lemma \ref{lem:taylor}, we have
$\phi_{I}(t) = \sum_{n \geq 2}b_{n}(a)p^{k(n - 2)}t^n$. Since the graph of $\phi$ is a normalized parabolic
arc, $b_{2}(a) \equiv \kappa\pmod{p}$
and hence $b_{2}(a)$ (viewed as the coefficient of $t^2$ in $\phi_{I}(t)$) is a unit in $\cO$.
On the other hand, for $n \geq 3$, $|p^{k(n - 2)}b_{n}(a)|_{p} \leq |b_{n}(a)|_{p} \leq \sup_{m \geq n}|c_{m}|_{p} \rightarrow 0$ as $n \rightarrow \infty$. Therefore $\phi_{I}$
satisfies all the conditions in Definition \ref{def:good-arc}. Writing $I = a + p^k\cO$
for some $a \in I$ gives $\Psi_{I}(t, \phi_{I}(t)) =(a + p^{k}t, \phi(a + p^{k}t))$ and hence
the graph of $\phi_{I}$ over $\cO$
is exactly $\Psi_{I}^{-1}(\Gamma|_{I})$.

For $(b)$, we compute 
\begin{align*}
\wh{\wt{F_{I}}}(\xi, \eta) = |\det L_{I}|_{p}\wh{F_I}((a, \phi(a)) + L_{I}(\xi, \eta)) = p^{-3k}\wh{F_{I}}(a + p^{k}\xi, \phi(a) + \phi'(a)p^{k}\xi + p^{2k}\eta).
\end{align*}
By where the Fourier support of $F_{I}$
is defined, we have $a + p^{k}\xi \in I$ and 
$|\phi(a) + \phi'(a)p^{k}\xi + p^{2k}\eta - \phi(a + p^{k}\xi)|_{p} \leq p^{-2N}$. The first condition implies $\xi \in \cO$.
Meanwhile, by the definition of $\phi_{I}$,
the second condition implies $|p^{2k}(\eta - \phi_{I}(\xi))|_{p} \leq p^{-2N}$ and hence $|\eta - \phi_{I}(\xi)|_{p} \leq p^{2(k - N)}$.
 \end{proof}

\begin{lemma}\label{lem:transversality}
Let $p$ be an odd prime and let $\phi$ be as in
Definition~\ref{def:good-arc}. If $\xi_1,\xi_2,\xi_3,\xi_4\in\cO$ satisfy
\[
\xi_1 + \xi_2 = \xi_3 + \xi_4,
\]
then
\[
  |\phi(\xi_1)+\phi(\xi_2)-\phi(\xi_3)-\phi(\xi_4)|_p
  =|\xi_3-\xi_1|_p\,|\xi_4-\xi_1|_p .
\]
\end{lemma}

\begin{proof}
Write $h := \xi_3 - \xi_1 = \xi_2 - \xi_4$ and $k:=\xi_4-\xi_1$. Then
\[
\xi_3 = \xi_1 + h, \quad \xi_4 = \xi_1+k, \quad \xi_2 = \xi_1 + h + k,   
\]
so
\[
  \phi(\xi_1)+\phi(\xi_2)-\phi(\xi_3)-\phi(\xi_4)=[\phi(\xi_1+h+k)-\phi(\xi_1)]
   -[\phi(\xi_1+h)-\phi(\xi_1)]
   -[\phi(\xi_1+k)-\phi(\xi_1)].
\]
Applying Lemma~\ref{lem:taylor} at $a=\xi_1$ to each of the three
brackets, with increments $k+h$, $h$ and $k$, respectively, and
subtracting the resulting convergent series termwise, we obtain
\[
  \phi(\xi_1)+\phi(\xi_2)-\phi(\xi_3)-\phi(\xi_4)=\sum_{n\geq1}b_n(\xi_1)\,[(k+h)^{n}-k^{n}-h^{n}].
\]
The bracket vanishes for $n=1$, equals $2k h$ for $n=2$, and for
$n\geq3$ the binomial theorem gives
\[
  (k+h)^{n}-k^{n}-h^{n}
  =\sum_{j=1}^{n-1}\binom{n}{j}k^{j}h^{\,n-j}
  =k h\,\gamma_n,
  \qquad
  \gamma_n:=\sum_{j=1}^{n-1}\binom{n}{j}k^{j-1}h^{\,n-j-1}\in\cO,
\]
every term of the sum having $1\leq j\leq n-1$. Hence
\[
  \phi(\xi_1)+\phi(\xi_2)-\phi(\xi_3)-\phi(\xi_4)=k h(2b_2(\xi_1)+\sum_{n\geq3}b_n(\xi_1)\gamma_n).
\]
The series converges because $|b_n(\xi_1)\gamma_n|_p\leq|b_n(\xi_1)|_p$
and this latter expression $\to0$
by inserting $h = 1$ into the conclusion of Lemma~\ref{lem:taylor}.

It remains to check that $2b_2(\xi_1)+\sum_{n\geq3}b_n(\xi_1)\gamma_n$ is a unit. By
Lemma~\ref{lem:taylor} we have $b_2(\xi_1)\equiv\kappa\pmod p$ with
$|\kappa|_p=1$, and $|2|_p=1$ as $p$ is odd, so $|2b_2(\xi_1)|_p=1$. On the other hand, $b_n(\xi_1)\in p\cO$ and $\gamma_n\in\cO$ for $n\geq3$, so the
ultrametric inequality bounds the tail by $|\sum_{n\geq3}b_n(\xi_1)\gamma_n|_p \leq p^{-1}<1$. This proves $|\phi(\xi_1)+\phi(\xi_2)-\phi(\xi_3)-\phi(\xi_4)|_p=|k|_p|h|_p$.
\end{proof}

A further ingredient we need is the following bilinear restriction estimate. First we observe:

\begin{lemma} \label{lem:deriv_expand}
Let $\phi$ be as in Definition~\ref{def:good-arc}. Then for all $s, t \in \cO$,
    $$|\phi'(s)-\phi'(t)|_p =|s-t|_p.$$
\end{lemma}
\begin{proof}
This is clear for $s = t$, so consider $s\ne t$. In that case, \eqref{eq:normal-form} and the integral power
series expansion show that $(\phi'(s)-\phi'(t))/(s-t)$ is a unit in $\cO$
congruent to $2\kappa \pmod{p}$, so the desired equality follows.
\end{proof}

\begin{proposition} [Bilinear restriction] \label{prop:bilinearrest}
Let $\phi$ be as in Definition~\ref{def:good-arc} and $\delta, \sigma \in p^{- \mathbb N}$. Suppose $J_1, J_2 \subset \Z_p$ are disjoint $p$-adic intervals at a distance $\sigma$ from each other, and suppose $f$ is Fourier supported on a $\delta$ neighborhood of $\{(t,\phi(t)) \colon t \in \Z_p\}$. 
If $p^{-k} := \delta/\sigma \leq \delta^{1/2}$ and $|J_1|, |J_2| \geq p^{-k}$, then
\begin{equation*}%
\int_{\Qp^2} |f_{J_1} f_{J_2}|^2 \leq \frac{1}{4} \int_{\Qp^2} \Big( \sum_{I \in \I_k} |f_I|^2 \Big)^2.
\end{equation*}
\end{proposition}

\begin{proof}  Partition $J_1$ and $J_2$ into $p$-adic intervals of length $p^{-k}$. Write
$f_{J_1} = \sum_{I_1 \subset J_1, I_1 \in \I_k} f_{I_1}$
where $f_{I_1}$ is the frequency localization of $f_{J_1}$ to the $\delta$ neighborhood of $\{(t,\phi(t)) \colon t \in I_1\}$, and similarly $f_{J_2} = \sum_{I_2 \subset J_2, I_2 \in \I_k} f_{I_2}$. 
If $I_1, I_2, I_3, I_4 \in \I_k$ with $I_1, I_3 \subset J_1$, $I_2, I_4 \subset J_2$ and $$\int_{\Qp^2} f_{I_1} f_{I_2} \overline{f_{I_3} f_{I_4}} \ne 0,$$ then there exist $\xi_1 \in I_1$, $\xi_2 \in I_2$, $\xi_3 \in I_3$, $\xi_4 \in I_4$ such that 
$\xi_1 + \xi_2 = \xi_3 + \xi_4$ and $|\phi(\xi_1) + \phi(\xi_2) - \phi(\xi_3) - \phi(\xi_4)|_p \leq \delta.$
From Lemma~\ref{lem:transversality}, this says
\[
|\xi_3 - \xi_1|_p |\xi_4 - \xi_1|_p \leq \delta.
\]
But since $I_1 \subset J_1$, $I_4 \subset J_2$, from the distance between $J_1$ and $J_2$ we get $|\xi_4 - \xi_1|_p \geq \sigma$. Thus the above inequality implies $|\xi_3 -\xi_1|_p \leq \delta/\sigma = p^{-k}$. Since both $\xi_1 \in I_1$, $\xi_3 \in I_3$ and $|I_1| = |I_3| = p^{-k}$, this forces
$I_1 = I_3$. Similarly $I_2 = I_4$. This allows one to expand
\begin{equation*} %
\int_{\Q_p^2} |f_{J_1} f_{J_2}|^2 =  \int_{\Q_p^2} \sum_{\substack{I_1 \in \I_k \\ I_1 \subset J_1}} \sum_{\substack{I_2 \in \I_k \\ I_2 \subset J_2}} |f_{I_1} f_{I_2}|^2.
\end{equation*}
Now use the pointwise inequality $ab \leq \frac{1}{4}(a+b)^2$ to conclude that
\[
\sum_{\substack{I_1 \in \I_k \\ I_1 \subset J_1}} \sum_{\substack{I_2 \in \I_k \\ I_2 \subset J_2}} |f_{I_1} f_{I_2}|^2 \leq \frac{1}{4} \Big( \sum_{I \in \I_k} |f_I|^2 \Big)^2.
\]
\end{proof}

\section{Orthogonal expansion of $F^3$}

\begin{lemma} \label{lem:pmod3}
If $p$ is an odd prime and $p \equiv 2 \pmod 3$, then 
\[
|r^2+rs+s^2|_p = \max\{|r|_p,|s|_p\}^2.
\]
\end{lemma}

\begin{proof}
The key is that if $p \equiv 2 \pmod 3$ then $-3$ is not a quadratic residue mod $p$. Quadratic reciprocity gives
\[
\Big( \frac{-3}{p} \Big) = \Big( \frac{-1}{p} \Big) \Big( \frac{3}{p} \Big) = (-1)^{(p-1)/2} (-1)^{(p-1)/2} \Big( \frac{p}{3} \Big) = \Big( \frac{2}{3} \Big) = -1.
\]
As a result, one has 
\begin{equation} \label{eq:consequencemod3}
|x^2 + 3 y^2|_p = \max\{|x|_p,|y|_p\}^2 \quad \text{for any $x, y \in \Q_p$}.
\end{equation}
Indeed, the ultrametric inequality says the left hand side is bounded by the right hand side. If the right hand side of \eqref{eq:consequencemod3} is normalized to $1$, then either $y \in p\cO$, in which case $\max\{|x|_p,|y|_p\} = 1$ implies $|x|_p = 1$ so $|x^2 + 3 y^2|_p = |x^2|_p = 1$; or else $y \in \cO^{\times}$, in which case for $1 > |x^2 + 3 y^2|_p = |(x/y)^2 + 3|_p$ to hold true one must have a non-trivial solution $z \in \mathbb F_p$ with $z^2 + 3 = 0$, contradicting the fact that $-3$ is not a quadratic residue mod $p$. Hence in the second case \eqref{eq:consequencemod3} holds as well. From \eqref{eq:consequencemod3}, if $p$ is also odd, one obtains
\[
|r^2 + rs + s^2|_p = |(r+s/2)^2 + 3(s/2)^2|_p = \max\{|r+s/2|_p,|s/2|_p\}^2 = \max\{|r|_p,|s|_p\}^2,
\]
as desired.
\end{proof}

\begin{lemma}\label{lem:cubic-separation}
Let $p$ be an odd prime with $p\equiv2\pmod3$ and let $\phi$ be as in
Definition \ref{def:good-arc}. Define $\Lambda \colon \cO^2 \to \cO$ by
$$\Lambda(\sigma,\tau):=\tau-3\,\phi(\sigma/3).$$
Then for any $\xi_1,\xi_2,\xi_3\in\cO$, if $(\sigma,\tau) :=(\xi_1,\phi(\xi_1))+(\xi_2,\phi(\xi_2))+(\xi_3,\phi(\xi_3))$,
one has
$|\Lambda(\sigma,\tau)|_p= \max_{i\ne j}|\xi_i-\xi_j|_p^2.$
\end{lemma}

\begin{proof}
Since $p\equiv2\pmod3$ we have $p\neq3$, so $3\in\cO^\times$; thus
$a:=\sigma/3\in\cO$. Put $u_i:=\xi_i-a$, so that $u_1+u_2+u_3=0$ and
$u_i=\frac{1}{3}\sum_{j=1}^3(\xi_i-\xi_j)$ which implies $|u_i|_p\leq\rho:=  \max_{i\ne j}|\xi_i-\xi_j|_p.$
By Lemma \ref{lem:taylor},
\[
  \Lambda(\sigma,\tau)=\sum_{i=1}^3\phi(a+u_i)-3\phi(a)
  =b_2(a)\sum_{i=1}^3u_i^{2}
   +\sum_{n\geq3}b_n(a)\sum_{i=1}^3u_i^{n},
\]
the terms for $n=1$ canceling because $\sum_iu_i=0$. We compute $|\Lambda(\sigma,\tau)|_p$ by analyzing the right hand side.

For the quadratic term $b_2(a)\sum_{i=1}^3u_i^{2}$, write $r:=\xi_1-\xi_2$ and $s:=\xi_2-\xi_3$, so that
$\xi_1-\xi_3=r+s$ and $\rho=\max \{|r|_p,|s|_p\}$. If $\rho=0$, all three parameters are equal and both
sides are zero; hence assume $\rho>0$. From
$3\sum_{i=1}^3 u_i^2 = (u_1 - u_2)^2 + (u_2 - u_3)^2 + (u_3 - u_1)^2 + (u_1 + u_2 + u_3)^2$ we get
\[
  \sum_{i=1}^3u_i^{2}
  =\frac13(r^2+s^2+(r+s)^2)
  =\frac23(r^2+rs+s^2).
\]
Since $p\equiv2\pmod3$, Lemma \ref{lem:pmod3} and $|2/3|_p=1$ give $|\sum_iu_i^2|_p=\rho^{2}$. As $|b_2(a)|_p=1$ by
Lemma \ref{lem:taylor}, the quadratic term has norm exactly $\rho^{2}$.

For the tail $\sum_{n\geq3}b_n(a)\sum_{i=1}^3u_i^{n}$, Lemma~\ref{lem:taylor} gives $|b_n(a)|_p\leq p^{-1}$ for
$n\geq3$, while $|\sum_iu_i^{n}|_p\leq\rho^{n}\leq\rho^{3}$
because $\rho\leq1$. Hence the tail has norm at most
$p^{-1}\rho^{3}<\rho^{2}$, and the ultrametric inequality yields
$|\Lambda(\sigma,\tau)|_p=\rho^{2}$.
\end{proof}

\begin{proposition}[Orthogonal expansion of $F^3$]\label{orthogonalexpansion}
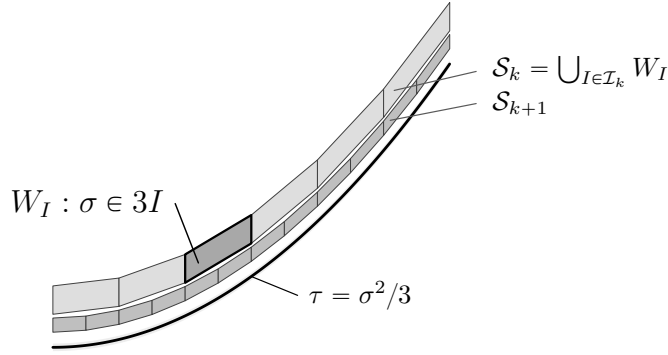
\begin{figure}
\begin{tikzpicture}[
  x=1.75cm,y=1.25cm,
  font=\small,
  line cap=round,
  line join=round,
  every node/.style={text=black},
  support cell/.style={draw=black!75,line width=.35pt},
]

  \path[fill=black!7,draw=none]
    plot[domain=0:3,samples=121]
      (\x,{\x*\x/3+0.040})
    --
    plot[domain=3:0,samples=121]
      (\x,{\x*\x/3-0.040})
    -- cycle;

  \foreach \j in {0,...,11}{
    \pgfmathsetmacro{\a}{0.25*\j}
    \pgfmathsetmacro{\b}{0.25*(\j+1)}
    \path[support cell,fill=black!25]
      (\a,{\a*\a/3+0.160})
      --
      (\b,{\b*\b/3+0.160})
      --
      (\b,{\b*\b/3+0.310})
      --
      (\a,{\a*\a/3+0.310})
      -- cycle;
  }

  \foreach \j in {0,...,5}{
    \pgfmathsetmacro{\a}{0.50*\j}
    \pgfmathsetmacro{\b}{0.50*(\j+1)}
    \path[support cell,fill=black!13]
      (\a,{\a*\a/3+0.350})
      --
      (\b,{\b*\b/3+0.350})
      --
      (\b,{\b*\b/3+0.650})
      --
      (\a,{\a*\a/3+0.650})
      -- cycle;
  }

  \path[
    draw=black,
    fill=black!34,
    line width=.85pt
  ]
    (1.00,{1.00*1.00/3+0.350})
    --
    (1.50,{1.50*1.50/3+0.350})
    --
    (1.50,{1.50*1.50/3+0.650})
    --
    (1.00,{1.00*1.00/3+0.650})
    -- cycle;

  \node[
    font=\normalsize\bfseries,
    anchor=north
  ] (WIlabel) at (0.26,1.78)
    {$W_I: \sigma \in 3I$};

  \draw[black,line width=.55pt]
    (WIlabel.east)
    -- (1.10,{1.10*1.10/3+0.49});

  \draw[black,line width=1.05pt]
    plot[domain=0:3,samples=121]
      (\x,{\x*\x/3});

\node[
  anchor=west,
  font=\footnotesize
] (parabolaLabel) at (1.85,0.55)
  {$\tau=\sigma^2/3$};

\draw[black,line width=.45pt]
  (parabolaLabel.west)
  -- (1.50,{1.50*1.50/3});

  \draw[black!65,line width=.45pt]
    (2.58,{2.58*2.58/3+0.500})
    -- (3.20,2.92);

  \node[
    anchor=west,
    font=\footnotesize
  ] at (3.24,2.92)
    {$\mathcal S_k = \bigcup_{I \in \I_k} W_I$};

  \draw[black!65,line width=.45pt]
    (2.55,{2.55*2.55/3+0.235})
    -- (3.20,2.58);

  \node[
    anchor=west,
    font=\footnotesize
  ] at (3.24,2.58)
    {$\mathcal S_{k+1}$};

\end{tikzpicture}
\caption{Fourier support of $T_I(F)$ as $I$ varies in the case $\phi(t) = t^2$}\label{tiffig}
\end{figure}
Let $p$ be an odd prime with $p \equiv 2 \pmod 3$. Let $F$ be as in Theorem \ref{thm:main}. Recall the telescoping expansion of $F^3$ in Section~\ref{sec:cubic}:
\begin{equation} \label{eq:intro-F3-telescope}
 F^3=\sum_{k=0}^{N-1}\sum_{I\in\I_k}T_I(F)
       +\sum_{K\in\I_N}F_K^3.
\end{equation}
\begin{enumerate}[(a)]
    \item For $0 \leq k < N$, the Fourier support of $\sum_{I \in \I_k} T_I(F)$ is contained in 
    \[
    \mathcal{S}_k := \{(\sigma,\tau) \in \cO^2 \colon |\Lambda(\sigma,\tau)|_p = p^{-2k}\}
    \]
    whereas the Fourier support of $\sum_{K \in \I_N} F_K^3$ is contained in $\{(\sigma,\tau) \in \cO^2 \colon |\Lambda(\sigma,\tau)|_p \leq p^{-2N}\}$.
    In particular, these sums have pairwise disjoint Fourier supports. 
    \item For $0 \leq k < N$, the functions $\{T_I(F)\}_{I \in \I_k}$ have pairwise disjoint Fourier supports. Similarly, the functions $\{F_K^3\}_{K \in \I_N}$ have pairwise disjoint Fourier supports. See Figure \ref{tiffig} for a picture.
    \item We have
    \[
    \|F\|_6^6 = \sum_{k=0}^{N-1} \sum_{I \in \I_k} \|T_I(F)\|_2^2 + \sum_{K \in \I_N} \|F_K\|_6^6.
    \]
\end{enumerate}
\end{proposition}

\begin{proof}
    \begin{enumerate}[(a)]
        \item Fix $0 \leq k < N$. The Fourier support of $\sum_{I \in \I_k} T_I(F)$ is contained in 
        \[
        \bigcup_{I \in \I_k} \bigcup_{(J_1,J_2,J_3) \in C_3(I)} \text{supp}\, \widehat{F_{J_1} F_{J_2} F_{J_3}}.
        \]
        If $I \in \I_k$ and $(J_1,J_2,J_3) \in C_3(I)$, a point $(\sigma,\tau)$ in the Fourier support of $F_{J_1} F_{J_2} F_{J_3}$ can be written in the form $(\xi_1,\phi(\xi_1) + \epsilon_1) + (\xi_2, \phi(\xi_2) + \epsilon_2) + (\xi_3, \phi(\xi_3)+\epsilon_3)$ where $\xi_i \in J_i$ and all $|\epsilon_i|_p \leq p^{-2N}$. Recall $\Lambda(\sigma,\tau) = \tau - 3 \phi(\sigma/3)$ is linear in $\tau$. Thus 
        \[
 \Lambda(\sigma,\tau)
 =\Lambda\!\Big(\sum_i (\xi_i,\phi(\xi_i)) \Big)+\sum_i\epsilon_i,
 \qquad
 |\sum_i\epsilon_i|_p\le p^{-2N}<p^{-2k}.
\]
By Lemma~\ref{lem:cubic-separation} we have 
$|\Lambda(\sigma,\tau)|_p = \max_{i \ne j} |\xi_i-\xi_j|_p^2 = p^{-2k}.$
        Similarly, the Fourier support of $\sum_{K \in \I_N} F_K^3$ is contained in $\bigcup_{K \in \I_N} \text{supp}\, \widehat{F_K^3}$. If $K \in \I_N$ and $(\sigma,\tau)$ is in the Fourier support of $F_K^3$, then the above argument shows that $|\Lambda(\sigma,\tau)|_p \leq p^{-2N}$.
        \item Just observe that if $0 \leq k < N$ and $I \in \I_k$, then the Fourier support of $T_I(F)$ is contained in $$W_I := \{(\sigma,\tau) \in \mathcal{S}_k \colon \sigma \in 3I\},$$ and the sets $3I := I + I + I$ are disjoint as $I$ varies over $\I_k$ since $3$ is a unit in $\cO$. The projection onto the first coordinate of the Fourier support of $F_K^3$ is contained in $3K$, and the sets $3K$ are disjoint as $K$ varies over $\I_N$.
        \item From \eqref{eq:intro-F3-telescope}, we have
        \[
        \|F\|_6^6 = \sum_{k=0}^{N-1} \| \sum_{I \in \I_k} T_I(F) \|_2^2 + \| \sum_{K \in \I_N} F_K^3 \|_2^2 =  \sum_{k=0}^{N-1} \sum_{I \in \I_k} \|T_I(F)\|_2^2 + \sum_{K \in \I_N} \|F_K\|_6^6
        \]
        by successively applying (a) and (b).        
    \end{enumerate}
\end{proof}

\section{Induction on scales}%

Fix $m \geq 1$. The goal will be to prove that for $C_p$ large enough (see \eqref{eq:Cpchoice} below for the precise choice), if at most $2^m$ of the terminal pieces $F_K$ are nonzero and $F$ is normalized so that
\begin{equation} \label{eq:Anormalization}
\sum_{K \in \I_N} \|F_K\|_{\infty}^2 = 1,
\end{equation}
then 
\[
\|F\|_6^6 \leq C_p m \|F\|_2^2.
\]
This will be done by induction on the scale $N$ found in the statement of Theorem \ref{thm:main}. 

The base case $N = 0$ follows from the H\"{o}lder estimate $\norm F_6^6\leq\norm F_\infty^4\norm F_2^2$.
The rest of the proof is devoted to the inductive step. We assume $N \geq 1$ and the conclusion holds with $N$ replaced by all nonnegative integers $< N$. Even if $F$ is normalized as in \eqref{eq:Anormalization}, its constituent pieces $F_I$ may not be. For $0 \leq k \leq N$ and $I \in \I_k$, set
\begin{align*}%
A_{I} := \sum_{K \in \mc{I}_N, K \subset I}\|F_{K}\|_{\infty}^{2}, \quad E_{I} := \|F_{I}\|_{2}^{2}, \quad f_{I} := \begin{cases}
A_{I}^{-1/2}F_{I} & \text{ if } A_{I} > 0\\
0 & \text{ if } A_I = 0.
\end{cases}
\end{align*}
If $A_I = 0$, note that $F_I =0 $ and $E_I = 0$. This also gives the identities
$F_{I} = A_{I}^{1/2}f_I$ and $A_{I}\|f_I\|_{2}^{2} = E_{I}$.
For $A_{I} > 0$, we then have $\sum_{K \in \mc{I}_N, K \subset I}\|(f_I)_K\|_{\infty}^{2} = 1$.
Finally, for non-terminal $I$, we have
$A_{I} = \sum_{J \in C(I)}A_J$ and $E_{I} = \sum_{J \in C(I)}E_{J}$.

\subsection{Broad-narrow decompositions}
Recall that for $0 \leq k < N$ and $I \in \mc{I}_k$, $C_{3}(I)$ is the set
of off-diagonal triples of children of $I$.
 We expand
\begin{align*}
\|T_I(F)\|_2^2 &= \int_{\Q_{p}^2} |  \sum_{(J_1,J_2,J_3) \in C_3(I)} F_{J_1} F_{J_2} F_{J_3} |^2 \leq p^3 \sum_{(J_1,J_2,J_3) \in C_3(I)}  \int_{\Q_{p}^2} |F_{J_1} F_{J_2} F_{J_3}|^2. 
\end{align*}

Fix ${\bf J} = (J_1,J_2,J_3) \in C_3(I)$. Let $\B := \B_{\bf J}$ be the broad set given by
\begin{equation} \label{eq:broadsetdef}
\B = \{x \colon |f_{J_1}(x)|^6 + |f_{J_2}(x)|^6 + |f_{J_3}(x)|^6  \leq 6 p^3 |f_{J_1}(x) f_{J_2}(x) f_{J_3}(x)|^2 \}.
\end{equation}
Let $\cN := \cN_{\bf J}$ be the narrow set $\Q_p^2 \setminus \B_{\bf J}$. 
We estimate
\[
\int_{\Q_{p}^2} |f_{J_1} f_{J_2} f_{J_3}|^2 = \int_{\B} |f_{J_1} f_{J_2} f_{J_3}|^2 + \int_{\cN} |f_{J_1} f_{J_2} f_{J_3}|^2.
\]
On the narrow set 
\[
|f_{J_1}(x) f_{J_2}(x) f_{J_3}(x)|^2 \leq \frac{1}{6 p^3} (|f_{J_1}(x)|^6 + |f_{J_2}(x)|^6 + |f_{J_3}(x)|^6)
\]
so 
\begin{equation} \label{eq:narrow}
\begin{split}
\int_{\cN} |f_{J_1} f_{J_2} f_{J_3}|^2 &\leq \frac{1}{6 p^3} ( \int_{\Q_{p}^2} |f_{J_1}|^6 + \int_{\Q_{p}^2} |f_{J_2}|^6 + \int_{\Q_{p}^2} |f_{J_3}|^6 )\\
&\leq \frac{C_p}{6 p^3} m ( \|f_{J_1}\|_2^2 + \|f_{J_2}\|_2^2 + \|f_{J_3}\|_2^2).
\end{split}
\end{equation}
In this last inequality, we first parabolically rescale back to the unit scale (enabled by Proposition~\ref{prop:rescalinggeom}) and then apply the inductive hypothesis with $N$ replaced by $N-k-1$.

Below we will use the high-low decomposition to prove a broad estimate:

\begin{proposition}[Broad estimate] \label{prop:broad} 
There exists a constant $c_p \geq 1$ such that if $\B = \B_{\bf J}$ is the broad set \eqref{eq:broadsetdef} associated to a triple ${\bf J} \in C_3(I)$ for some $I$, then
\begin{equation} \label{eq:trilinear_broad}
\int_{\B} |f_{J_1} f_{J_2} f_{J_3}|^2 \leq c_p m ( \|f_{J_1}\|_2^2 + \|f_{J_2}\|_2^2 + \|f_{J_3}\|_2^2).
\end{equation}
\end{proposition}

Combining the proposition with \eqref{eq:narrow}, we obtain
\[
\int_{\Q_{p}^2}  |f_{J_1} f_{J_2} f_{J_3}|^2 \leq ( \frac{C_p}{6p^3} + c_p ) m ( \|f_{J_1}\|_2^2 + \|f_{J_2}\|_2^2 + \|f_{J_3}\|_2^2).
\]
Hence
\begin{align*}
\|T_I(F)\|_2^2 \leq \, & ( \frac{C_p}{6} + p^3 c_p ) m \sum_{(J_1,J_2,J_3) \in C_3(I)} A_{J_1} A_{J_2} A_{J_3} ( \|f_{J_1}\|_2^2 + \|f_{J_2}\|_2^2 + \|f_{J_3}\|_2^2)  \\
= \, & ( \frac{C_p}{6} + p^3 c_p ) m \sum_{(J_1,J_2,J_3) \in C_3(I)}  ( E_{J_1} A_{J_2} A_{J_3} + A_{J_1} E_{J_2} A_{J_3} + A_{J_1} A_{J_2} E_{J_3} ) \\
= \, & ( \frac{C_p}{2} + 3 p^3 c_p ) m \sum_{(J_1,J_2,J_3) \in C_3(I)}  A_{J_1} A_{J_2} E_{J_3} \\
= \, & ( \frac{C_p}{2} + 3 p^3 c_p ) m ( A_I^2 E_I - \sum_{J \in C(I)} A_J^2 E_J ).
\end{align*}
Applying Proposition \ref{orthogonalexpansion}(c) and observing that the sum in $k$ now telescopes gives
\begin{align*}
\|F\|_6^6 & \leq ( \frac{C_p}{2} + 3 p^3 c_p ) m ( \|F\|_2^2 - \sum_{I \in \I_N} \|F_I\|_{\infty}^4 \|F_I\|_2^2 ) + \sum_{I \in \I_N} \|F_I\|_6^6  \leq ( \frac{C_p}{2} + 3 p^3 c_p ) m \|F\|_2^2.
\end{align*}
The last inequality follows from H\"{o}lder: $\|F_I\|_6^6 \leq \|F_I\|_{\infty}^4 \|F_I\|_2^2$, and the fact that $( \frac{C_p}{2} + 3 p^3 c_p) m \geq 1$.
Recall $c_p$ is the fixed constant from Proposition~\ref{prop:broad}. Induction closes provided
\begin{equation} \label{eq:Cpchoice}
C_p \geq 6 p^3 c_p,
\end{equation} 
because then the right hand side is at most $C_p m \|F\|_2^2$.
    
\section{The broad estimate}%

\subsection{Dyadically pigeonholing the height}\label{dyadicpigeon}

Our goal is now to prove the broad estimate, Proposition~\ref{prop:broad}. By relabeling the $J_i$'s we may assume $J_1 \ne J_2$, and by parabolic rescaling using Proposition~\ref{prop:rescalinggeom} we only need to prove \eqref{eq:trilinear_broad} with $I = \Z_p$.

Fix ${\bf J} := (J_1, J_2, J_3)$ in $\I_1^3$ with $J_1 \ne J_2$. Let $f = \sum_{J \in \{J_1, J_2, J_3\}} f_J$ (so for example if $J_3 = J_1$ then $f = f_{J_1} + f_{J_2}$). 
Thus for $K \subset J_i$ with $A_{J_i} > 0$, we have
$f_{K} = (f_{J_i})_{K} = A_{J_i}^{-1/2}F_K$.
By the normalization of the $f_{J_i}$, we have $\sum_{K \in \mc{I}_N, K \subset J_i}\|f_{K}\|_{\infty}^{2} \leq 1$ with equality whenever $f_{J_i} \neq 0$ and hence
\[
\sum_{K \in \I_N} \|f_K\|_{\infty}^2 \leq 3.
\]

We want to estimate $\int_{\B} |f_{J_1} f_{J_2} f_{J_3}|^2.$
If $|f_{J_1}f_{J_2}f_{J_3}| \leq 4^3$, then $$|f_{J_1}f_{J_2}f_{J_3}|^{2} \leq 4^4 |f_{J_1}f_{J_2}f_{J_3}|^{2/3} \leq \frac{4^4}{3}\sum_{i = 1}^{3}|f_{J_i}|^{2}.$$
Therefore
$$\int_{\{x \in \B \colon |f_{J_1}f_{J_2}f_{J_3}| \leq 4^3\}}|f_{J_1}f_{J_2}f_{J_3}|^{2} \leq \frac{4^4}{3}\sum_{i = 1}^{3}\|f_{J_i}\|_{L^2}^{2}.$$

For dyadic $\alpha \geq 4$, let
\begin{align*}%
U_{\alpha} := \{x \in \B \colon \alpha^3 < |f_{J_1}(x) f_{J_2}(x) f_{J_3}(x)| \leq (2\alpha)^3\}.
\end{align*}
Since at most $2^m$ terminal functions $f_K$ are nonzero, Cauchy-Schwarz and our normalization for $f_{J_i}$ give
\[
|f_{J_i}| \leq 2^{m/2} ( \sum_{\substack{K \in \I_N \\ K \subset J_i}} |f_K|^2 )^{1/2} \leq 2^{m/2}.
\]
Hence there are only $O(m)$ dyadic values of $\alpha \geq 4$ for which $U_{\alpha}$ is nonempty. The desired estimate \eqref{eq:trilinear_broad} will follow if one proves
\begin{equation} \label{eq:levelsetestSect5}
\alpha^6 |U_{\alpha}| \lesssim_p ( \|f_{J_1}\|_2^2 + \|f_{J_2}\|_2^2 + \|f_{J_3}\|_2^2).
\end{equation}
We note that this sum over $\alpha$ is the only contribution to the log in Theorem~\ref{thm:main}.

\subsection{Pruning}

From now on, fix a dyadic $\alpha \geq 4$. Let $L > 0$ be a height to be chosen later.

Define $f_{N+1,K} := f_K$, $f_{N,K} := f_{N+1,K} 1_{|f_{N+1,K}| \leq L}$ for $K \in \I_N$, and recursively, for $1 \leq k < N$ and $I \in \I_k$, define 
\begin{align}\label{fkidef}
f_{k+1,I} := \sum_{J \in C(I)} f_{k+1,J}, \quad f_{k,I} := f_{k+1,I} 1_{|f_{k+1,I}| \leq L}.
\end{align}
The lemma below implies that for $I \in \mc{I}_k$, $f_{k + 1, I}$ and $f_{k, I}$ have Fourier supports in $\Xi_I$.

\begin{lemma}
Let $I \in \mc{I}_k$ with $k \leq N$. If $G$ is a function with Fourier transform supported on $\Xi_{I}$, then $G \, 1_{|G| \leq L}$ also has
Fourier transform supported on $\Xi_{I}$.
\end{lemma}
\begin{proof}
This is just an application of wavepacket decomposition.
Since $|G|$ is constant on each translate $T \in \mathbb T_I$, amplitude truncation either retains or discards each wavepacket $G \, 1_T$. Each retained wavepacket has Fourier support in $\Xi_I$, so $G \, 1_{{|G|\le L}}$ does as well.
\end{proof}

For $1 \leq k \leq N$ and $I \in \I_k$, let 
\begin{align}\label{ekidef}
e_{k,I} := f_{k+1,I} - f_{k,I}
\end{align}
be the discarded part. We have 
\[
f_{k+1,I} = f_{k,I} + e_{k,I}
\]
as a decomposition into two functions with disjoint support; in particular, we have a pointwise identity
\begin{equation} \label{eq:ptwiseortho}
|f_{k+1,I}|^2 = |f_{k,I}|^2 + |e_{k,I}|^2.
\end{equation}
Note that the Fourier support of $|f_{k+1,I}|^2$, $|f_{k,I}|^2$, and $|e_{k,I}|^2$ is contained in $\Xi_{I} - \Xi_{I} = \{(\xi, \eta): |\xi|_{p} \leq p^{-k}, |\eta - \phi'(a)\xi|_{p} \leq p^{-2k}\}$.

We will write $f_k := \sum_{I \in \I_k} f_{k,I}$ for $1 \leq k \leq N$, and $f_{N+1} := f$. 
For $1 \leq k \leq N$, define the residual square function
\begin{align}\label{rkdef}
r_k := \sum_{I \in \I_k} |e_{k,I}|^2.
\end{align}
The compensated square function is
\begin{align}\label{corrsquare}
g_k := \sum_{I \in \I_k} |f_{k+1,I}|^2 + \sum_{j=k+1}^N r_j.
\end{align}
In particular, $g_N = \sum_{K \in \I_N} |f_K|^2 \leq 3$.

Recall $f_{k+1,J_i}$ is the frequency localization of $f_{k+1}$ to $J_i \times \Q_p$. For $i = 1,2,3$ and $0 \leq k < N$ we have a pointwise bound
\begin{align}\label{pruneapprox}
|f_{J_i} - f_{k+1,J_i}| \leq \sum_{j=k+1}^N \sum_{\substack{I \in \I_j \\ I \subset J_i}} |e_{j,I}| \leq \frac{1}{L} \sum_{j=k+1}^N r_j \leq \frac{1}{L} g_{k+1}.
\end{align}
Indeed, by \eqref{corrsquare}, the last inequality
is equivalent to showing $r_{k + 1} \leq \sum_{I \in \mc{I}_{k + 1}}|f_{k + 2, I}|^{2}$ which itself follows from \eqref{eq:ptwiseortho} and \eqref{rkdef}.

\subsection{Stopping time at every point}
Decompose $\Q_{p}^2 = \mathcal{L} \cup \bigcup_{k = 1}^{N - 1}\Om_k$ where
\begin{align*}%
\Om_k &:= \{x: g_{k}(x) > 2g_{N}(x)\} \cap \bigcap_{j = k + 1}^{N - 1}\{x: g_{j}(x) \leq 2g_{N}(x)\}
\end{align*}
and
\begin{align*}%
\mathcal{L} &:= \bigcap_{j = 1}^{N - 1}\{x: g_{j}(x) \leq 2g_{N}(x)\}.
\end{align*}
Choose our pruning height to be $$L := 24p/\alpha.$$

For $x \in \Om_k$, notice that $g_{k+ 1}(x) \leq 6$.
Indeed $g_{k + 1}(x) \leq 2g_{N}(x) \leq 6$ if $k + 1 \leq N - 1$. Meanwhile, if $k = N - 1$, $g_{k + 1}(x) = g_{N}(x) \leq 3$.
Additionally, if $x \in \mc{L}$, we have
$g_{1}(x) \leq 6$.
By \eqref{pruneapprox} and our choice of $L$, we have
\begin{align}
|f_{J_i}(x) - f_{k + 1, J_i}(x)| &\leq \frac{6}{L} \leq \frac{\alpha}{4p}, \quad\quad x \in \Om_k\label{approx1}\\
|f_{J_i}(x) - f_{1, J_i}(x)| &\leq \frac{6}{L} \leq \frac{\alpha}{4p}, \quad\quad x \in \mc{L} \label{approx2}
\end{align}
for each $i = 1, 2, 3$.

For $x \in U_{\alpha}$, we will prove that $|f_{J_i}(x)| \sim_p \alpha$ for each $i=1,2,3$. More precisely, we have
\begin{align}\label{indbd}
\frac{\alpha}{6^{1/3}p} \leq |f_{J_i}(x)| \leq 2 \cdot 6^{1/6}p^{1/2}\alpha.
\end{align}
Indeed, for $x \in U_\alpha \subset \mathcal B$, broadness gives
  \[
  |f_{J_i}(x)|
  \leq 6^{1/6}p^{1/2}
  |f_{J_1}(x)f_{J_2}(x)f_{J_3}(x)|^{1/3}.
  \]
  Applying this estimate to the other two factors in the product also gives
\[
|f_{J_i}(x)| \geq \frac{|f_{J_1}(x) f_{J_2}(x) f_{J_3}(x)|}{(6 p^3 |f_{J_1}(x) f_{J_2}(x) f_{J_3}(x)|^2)^{2/6} } = \frac{1}{6^{1/3} p} |f_{J_1}(x) f_{J_2}(x) f_{J_3}(x)|^{1/3}.
  \]
Since $\alpha < |f_{J_1}(x)f_{J_2}(x)f_{J_3}(x)|^{1/3} \leq 2\alpha$ on $U_\alpha$, \eqref{indbd} follows.

Since $\alpha \geq 4$, $U_{\alpha} \cap \mc{L} = \emptyset$. Indeed, for $x \in U_{\alpha} \cap \mc{L}$, we have from \eqref{eq:ptwiseortho} that $|f_{1, J_i}(x)|^{2} \leq |f_{2, J_i}(x)|^2 \leq g_1(x) \leq 6$
and by \eqref{approx2}, we have
\begin{align}\label{fjiupper}
|f_{J_i}(x)| \leq |f_{1, J_i}(x)| + \alpha/(4p) \leq \sqrt{6} + \alpha/4
\end{align}
for each $i$.
On the other hand since $x \in U_{\alpha}$,
$\alpha< |f_{J_1}(x)f_{J_2}(x)f_{J_3}(x)|^{1/3}$. Taking
the geometric mean of \eqref{fjiupper} gives
$\alpha < \sqrt{6} + \alpha/4$ which is not possible
since $\alpha \geq 4$.

For $x \in U_{\alpha} \cap \Om_k$, \eqref{indbd}
gives
\begin{align*}
\frac{\alpha}{4p} = \frac{6^{1/3}}{4} \cdot \frac{\alpha}{6^{1/3}p} < \frac{1}{2}|f_{J_i}(x)|.
\end{align*}
Combining this with \eqref{approx1} gives
\begin{align*}
|f_{k + 1, J_i}(x)| \geq |f_{J_i}(x)| - \frac{\alpha}{4p} > \frac{1}{2}|f_{J_i}(x)|
\end{align*}
for each $i=1,2$. Together with $|f_{J_3}(x)|^{2} \leq 8p\alpha^2$ from the upper bound in \eqref{indbd}, we have
\begin{align*}
\alpha^6 < |f_{J_1}(x)f_{J_2}(x)f_{J_3}(x)|^{2} \leq 128p\alpha^{2}|f_{k + 1, J_1}(x)f_{k + 1, J_2}(x)|^{2}
\end{align*}
for $x \in U_{\alpha} \cap \Om_k$.
Disjointness of the $\Om_k$ immediately gives
\begin{equation*}
\alpha^6 |U_{\alpha}| \leq 128 p \alpha^2 \sum_{k=1}^{N-1} \int_{\Omega_k} |f_{k+1,J_1} f_{k+1,J_2}|^2.
\end{equation*}

\subsection{Applying bilinear restriction}
For $1 \leq k < N$, $\Omega_k$ is a union of squares $Q$ of side length $p^k$. We apply Proposition~\ref{prop:bilinearrest} with $\sigma = 1$ and $\delta = p^{-k}$ to the localization $f_{k+1} 1_Q$ to each such square $Q$. %
Summing over all such squares $Q$ that make up $\Omega_k$ gives
\[
\int_{\Omega_k} |f_{k+1,J_1} f_{k+1,J_2}|^2 \leq \frac{1}{4} \int_{\Omega_k} g_k^2.
\]
As a result,
\begin{equation} \label{eq:high-low-integrals}
\alpha^6 |U_{\alpha}| \leq 
32 p \alpha^2  \sum_{k=1}^{N-1} \int_{\Omega_k} g_k^2.
\end{equation}
To establish \eqref{eq:levelsetestSect5} we estimate the right hand side of \eqref{eq:high-low-integrals}.

\subsection{The (new) High and Low Lemmas}

The compensated square function obeys the following new low lemma.

\begin{lemma}[Low Lemma] \label{lem:low}
Let $P_k$ be the frequency projection onto $B_{\rho_k}$, the ball of radius $\rho_k := p^{-k}$ centered at the origin.
Then for $1\leq k<N$,
\begin{equation*}%
 P_{k+1} g_k=g_{k+1}.
\end{equation*}
\end{lemma}

\begin{proof}
For $I\in\I_k$, writing $f_{k + 1, I} = \sum_{J \in C(I)}f_{k + 1, J}$ and expanding the square gives
\begin{equation} \label{eq:oldlow}
 P_{k+1}|f_{k+1,I}|^2
 =\sum_{J\in C(I)}|f_{k+1,J}|^2.
\end{equation}
Indeed, the diagonal terms have Fourier support in
$B_{\rho_{k+1}}$, while every cross term has Fourier support
outside that ball. 
(So far, this is exactly the argument for the old low lemma; see, for example \cite[Lemma 7.1]{GuoLiYung2023}.) Taking \eqref{eq:ptwiseortho} with $k$ replaced by $k+1$ shows that the right hand side of \eqref{eq:oldlow} is equal to 
$\sum_{J \in C(I)} (|f_{k+2,J}|^2 - |e_{k+1,J}|^2).$
Since $\supp\widehat r_j\subset B_{\rho_j}$, $P_{k+1}$ fixes $r_j$ for $j\geq k+1$. This shows
\begin{align*}
P_{k+1} g_k = \sum_{J \in \I_{k+1}} |f_{k+2,J}|^2 - r_{k+1} + \sum_{j={k+1}}^N r_j = g_{k+1}.
\end{align*}
\end{proof}

For $1 \leq k < N$, write the high-frequency part of $g_k$ as 
\begin{align*}%
h_k := (I-P_{k+1}) g_k = (I-P_{k+1}) \sum_{I \in \I_k} |f_{k+1,I}|^2.
\end{align*}
From the above Low Lemma, we have
$g_k = g_{k+1} + h_k $
which iterates to an $L^2$ orthogonal decomposition
\begin{equation} \label{eq:hjtelescope}
g_k = g_N + \sum_{j=k}^{N-1} h_j.
\end{equation}

\begin{lemma}[High Lemma]\label{lem:high}
Let 
\begin{align*}%
H := \sum_{k=1}^{N-1} h_k.
\end{align*}
Then
\begin{equation*}%
 \norm H_2^2 = \sum_{k=1}^{N-1} \|h_k\|_2^2 \leq p^2 L^2
 \sum_{i=1}^3\norm{f_{J_i}}_2^2.
\end{equation*}
\end{lemma}

\begin{proof}
Since $\supp\widehat g_k\subset B_{\rho_k}$, the functions $h_k$ have Fourier supports contained in pairwise disjoint annuli, leading to $\|H\|_2^2 = \sum_{k=1}^{N-1} \|h_k\|_2^2$. For $1 \leq k < N$ and $I\in\I_k$ let
\[
 h_{k}^{I}:=(I-P_{k+1})|f_{k+1,I}|^2.
\]
The Fourier support of $h_{k}^{I}$ is contained in the set 
$\{(\xi,\eta) \colon |\xi|_p = p^{-k}, |\eta - \phi'(a) \xi|_p \leq p^{-2k}\}$
where $a$ is any point in $I$. These sets are disjoint as $I$ varies over $\I_k$. If $I_1, I_2 \in \I_k$ and $I_1 \ne I_2$, then for $a_1 \in I_1$, $a_2 \in I_2$ we have $|\phi'(a_1) - \phi'(a_2)|_p = |a_1-a_2|_p \geq p^{-(k-1)}$ by Lemma~\ref{lem:deriv_expand}. 
Suppose there was a point $(\xi, \eta) \in (\supp \wh{h_{k}^{I_1}}) \cap (\supp \wh{h_{k}^{I_2}})$.
Then
\[
p^{-(2k-1)} \leq |\phi'(a_1)-\phi'(a_2)|_p |\xi|_p = |(\eta - \phi'(a_1) \xi) - (\eta - \phi'(a_2) \xi)|_p \leq p^{-2k},
\]
a contradiction. Thus the  Fourier supports of the $h_{k}^{I}$ are disjoint as $I$ varies over $\I_k$.

The old Low Lemma
\eqref{eq:oldlow} allows us to rewrite
\[
h_{k}^{I} = |f_{k+1,I}|^2 - \sum_{J \in C(I)} |f_{k+1,J}|^2.
\]
For each $I \in \I_k$, the Fourier support of $h_{k}^{I}$ is disjoint from the Fourier support of $\sum_{J \in C(I)} |f_{k+1,J}|^2$. In fact, the former is contained in $|\xi|_p = p^{-k}$, while the latter is contained in $|\xi|_p \leq |(\xi,\eta)|_p \leq p^{-(k+1)}$. 

The two geometric facts about Fourier supports imply
\begin{align*}
\|h_k\|_2^2 = \sum_{I \in \I_k} \|h_{k}^{I}\|_2^2 = \sum_{I \in \I_k} ( \|f_{k+1,I}\|_4^4 - \| \sum_{J \in C(I)} |f_{k+1,J}|^2 \|_2^2 ).
\end{align*}
Expanding the $L^2$ norm of the sum over $J$ and discarding the nonnegative cross terms, 
one obtains $$\| \sum_{J \in C(I)} |f_{k+1,J}|^2 \|_2^2 \geq \sum_{J \in C(I)}\nms{f_{k + 1, J}}_{4}^{4}.$$
Combining the above two centered equations with integrating the pointwise equality $|f_{k+1,I}|^4 = |f_{k,I}|^4 + |e_{k,I}|^4$ gives
\begin{align*}
\|h_k\|_2^2 
&\leq  \sum_{I \in \I_k} ( \|f_{k+1,I}\|_4^4 - \sum_{J \in C(I)} \| f_{k+1,J}\|_4^4 ) = \sum_{I \in \I_k} \|e_{k,I}\|_4^4 + \sum_{I \in \I_k} \|f_{k,I}\|_4^4 - \sum_{J \in \I_{k+1}} \| f_{k+1,J}\|_4^4.
\end{align*}
The last two terms form a telescoping sum in $k$. Hence
\[
\sum_{k=1}^{N-1} \|h_k\|_2^2 \leq \sum_{k=1}^{N-1} \sum_{I \in \I_k} \|e_{k,I}\|_4^4 + \sum_{I \in \I_1} \|f_{1,I}\|_4^4.
\]
For $I \in \mc{I}_1$, observe that $|f_{1, I}| = |f_{2, I}1_{|f_{2, I}| \leq L}| \leq L$
and for $I \in \mc{I}_k$, observe that
\begin{align*}
|e_{k, I}| \leq |f_{k + 1, I}| \leq \sum_{J \in C(I)}|f_{k + 1, J}| = \sum_{J \in C(I)}|f_{k + 2, J}1_{|f_{k + 2, J}| \leq L}| \leq pL
\end{align*}
for $1 \leq k < N$.
Therefore we have
\[
\|H\|_2^2 \leq p^2 L^2 \sum_{k=1}^{N-1} \sum_{I \in \I_k} \|e_{k,I}\|_2^2 + L^2 \sum_{I \in \I_1} \|f_{1,I}\|_2^2.
\]
Using $|e_{k,I}|^2 = |f_{k+1,I}|^2 - |f_{k,I}|^2$
and integrating, together with $L^2$ orthogonality, one obtains
\begin{equation*} %
\sum_{k=1}^{N} \sum_{I \in \I_k} \|e_{k,I}\|_2^2 = \sum_{k=1}^{N} (\|f_{k+1}\|_2^2 - \|f_k\|_2^2) = \|f\|_2^2 - \|f_1\|_2^2.
\end{equation*}
Finally since $\sum_{I \in \I_1} \|f_{1,I}\|_2^2 = \|f_1\|_2^2$, we have
\begin{align*}
\|H\|_2^2 &\leq p^2 L^2 (\|f\|_2^2 - \|f_1\|_2^2) + L^2 \|f_1\|_2^2 \leq p^2 L^2 \|f\|_2^2%
\end{align*}
and the conclusion follows from the definition of $f$ at the beginning of Section \ref{dyadicpigeon}.
\end{proof}

Recall we were estimating the right hand side of \eqref{eq:high-low-integrals}. On $\Omega_k$, we have
\[
g_k < 2 (g_k - g_N) = 2 \sum_{j=k}^{N-1} h_j = 2 P_k H.
\]
The first inequality follows from the definition of $\Omega_k$, the first equality from \eqref{eq:hjtelescope}, and the second equality from the observation that $\supp\wh{h}_j \subset B_{\rho_j}\bs B_{\rho_{j + 1}}$. 
As a result, for every $k = 1, 2, \ldots, N - 1$,
\[
\int_{\Omega_k} g_k^2 \leq 4 \int_{\Omega_k} |P_k H|^2 \leq 4\int_{\Om_k}|H|^2
\]
where in the last inequality we made use of \eqref{proj} and then Plancherel.
Using disjointness of $\Omega_k$, together with the High Lemma~\ref{lem:high} and our choice of $L$, we obtain
\begin{equation*}%
 \sum_{k=1}^{N-1}\int_{\Omega_k}g_k^2
 \leq4\norm H_2^2
 \leq 4 p^2 L^2\sum_{i=1}^3\norm{f_{J_i}}_2^2  = 4 \cdot 24^2 p^4 \alpha^{-2} \sum_{i=1}^3\norm{f_{J_i}}_2^2.
\end{equation*}
Plugging this back into \eqref{eq:high-low-integrals}, we obtain the desired conclusion \eqref{eq:levelsetestSect5}.

\end{document}